\documentclass[a4paper]{amsart}
\usepackage{amssymb,amscd}

\theoremstyle{problems}

\newtheorem{theorem}{Theorem}[section]
\newtheorem{proposition}[theorem]{Proposition}
\newtheorem{lemma}[theorem]{Lemma}
\newtheorem{corollary}[theorem]{Corollary}

\theoremstyle{definition}

\newtheorem{example}[theorem]{Example}
\newtheorem{construction}[theorem]{Construction}

\theoremstyle{remark}
\newtheorem*{remark}{Remark}

\numberwithin{equation}{section}

\renewcommand{\phi}{\varphi}

\def\R{\mathbb R}

\def\Z{\mathbb Z}

\def\sK{\mathcal K}

\newcommand{\mb}[1]{{\textbf {\textit#1}}}

\renewcommand{\ge}{\geqslant}

\newcommand{\Ker}{\mathop{\rm Ker}}
\newcommand{\rank}{\mathop{\mathrm{rank}}}

\def\raag{\mbox{\it RA\/}}
\def\racg{\mbox{\it RC\/}}

\newcommand{\cat}[1]{\mbox{\sc #1}}

\newcommand{\scat}[1]{\mbox{\scriptsize{\sc #1}}}

\def\ca{{\text{\sc cat}}}

\def\top{\text{\sc top}}

\def\colim{\mathop\mathrm{colim}\nolimits}
\def\lim{\mathop\mathrm{lim}\nolimits}
\def\hocolim{\mathop\mathrm{hocolim}\nolimits}

\newcommand{\fw}{\mathop{\mbox{\it FW\/}}}

\newcommand{\zk}{\mathcal Z_{\mathcal K}}
\newcommand{\lk}{\mathcal L_{\mathcal K}}
\newcommand{\rk}{\mathcal R_{\mathcal K}}

\def\pt{\mathit{pt}}

\begin{document}

\title[Polyhedral products and commutator subgroups]{Polyhedral
products and commutator subgroups of right-angled Artin and Coxeter groups}

\author[Taras Panov]{Taras Panov}
\address{Department of Mathematics and Mechanics, Moscow
State University, Leninskie Gory, 119991 Moscow, Russia,
\newline\indent Institute for Theoretical and Experimental Physics,
Moscow, Russia \quad \emph{and}
\newline\indent Institute for Information Transmission Problems,
Russian Academy of Sciences} \email{tpanov@mech.math.msu.su}
\urladdr{http://higeom.math.msu.su/people/taras/}

\author{Yakov Veryovkin}
\address{Department of Mathematics and Mechanics, Moscow
State University, Leninskie Gory, 119991 Moscow, Russia,
\newline\indent Steklov Mathematical Institute, Russian Academy of Sciences}
\email{verevkin\_j.a@mail.ru}

\thanks{The research of the first author was carried out at the
IITP RAS and supported by the Russian Science Foundation grant
no.~14-50-00150. The research of the second author was supported
by the Russian Foundation for Basic Research grant
no.~14-01-00537.}



\begin{abstract}
We construct and study polyhedral product models for classifying
spaces of right-angled Artin and Coxeter groups, general graph
product groups and their commutator subgroups. By way of
application, we give a criterion of freeness for the commutator
subgroup of a graph product group, and provide an explicit minimal
set of generators for the commutator subgroup of a right-angled
Coxeter group.
\end{abstract}

\dedicatory{To the memory of Rainer Vogt}

\maketitle

\section{Introduction}
Right-angled Artin and Coxeter groups are familiar objects in
geometric group theory~\cite{davi08}. From the abstract
categorical viewpoint, they are particular cases of \emph{graph
product groups}, corresponding to a sequence of $m$ groups $\mb
G=(G_1,\ldots,G_m)$ and a graph~$\Gamma$ on $m$ vertices.
Informally, the graph product group $\mb G^\Gamma$ consists of
words with letters from $G_1,\ldots,G_m$ in which the elements of
$G_i$ and $G_j$ with $i\ne j$ commute whenever $\{i,j\}$ is an
edge of~$\Gamma$. The graph product group $\mb G^\Gamma$
interpolates between the free product $G_1\star\cdots\star G_m$
(corresponding to a graph consisting of $m$ disjoint vertices) and
the cartesian product $G_1\times\cdots\times G_m$ (corresponding
to a complete graph). Right-angled Artin and Coxeter groups
$\raag_\Gamma$ and $\racg_\Gamma$ correspond to the cases $G_i=\Z$
and $G_i=\Z_2$, respectively.

The \emph{polyhedral product} is a functorial
combinatorial-topological construction assigning a topological
space $(\mb X,\mb A)^\sK$ to a sequence of $m$ pairs of
topological spaces $(\mb X,\mb A)=\{(X_1,A_1),\ldots,(X_m,A_m)\}$
and a simplicial complex $\sK$ on $m$
vertices~\cite{bu-pa00,b-b-c-g10,bu-pa15}. It generalises the
notion of a \emph{moment-angle complex} $\zk=(D^2,S^1)^\sK$, which
is a key object of study in toric topology. Polyhedral products
also provide a unifying framework for several constructions of
classifying spaces for right-angled Artin and Coxeter groups,
their commutator subgroups, as well as general graph products
groups. The description of the classifying spaces of graph product
groups and their commutator subgroups was implicit
in~\cite{p-r-v04}, where the canonical homotopy fibration
\[
  (E\mb G,\mb G)^\sK\longrightarrow (B\mb
  G)^\sK\longrightarrow\prod_{k=1}^m BG_k.
\]
of polyhedral products was introduced and studied.

To each graph $\Gamma$ without loops and double edges one can
assign a \emph{flag} simplicial  complex $\sK$, whose simplices
are the vertex sets of complete subgraphs (or \emph{cliques})
of~$\Gamma$. For any flag complex $\sK$ the polyhedral product
$(B\mb G)^\sK$ is the classifying space for the corresponding
graph product group $\mb G^\sK=\mb G^\Gamma$, while $(E\mb G,\mb
G)^\sK$ is the classifying space for the commutator subgroup
of~$\mb G^\sK$. In the case of right-angled Artin group
$\raag_\Gamma=\raag_\sK$, each $BG_i=B\Z$ is a circle, so we
obtain as $(B\mb G)^\sK$ the subcomplex $(S^1)^\sK$ in an
$m$-torus introduced by Kim and Roush in~\cite{ki-ro80}. In the
case of right-angled Coxeter group $\racg_\sK$, each $BG_i=B\Z_2$
is an infinite real projective space~$\R P^\infty$, so the
classifying space for $\racg_\sK$ is a similarly defined
subcomplex $(\R P^\infty)^\sK$ in the $m$-fold product of~$\R
P^\infty$. The classifying space for the commutator subgroup of
$\racg_\sK$ is a finite cubic subcomplex $\rk$ in an
$m$-dimensional cube, while the classifying space for the
commutator subgroup of $\raag_\sK$ is an infinite cubic subcomplex
$\lk$ in the $m$-dimensional cubic lattice. All these facts are
summarised in Theorem~\ref{gpfund} and Corollaries~\ref{artfund}
and~\ref{coxfund}.

The emphasis of~\cite{p-r-v04} was on properties of graph products
of topological (rather than discrete) groups, as part of the
homotopy-theoretical study of toric spaces and their loop spaces.
In the present work we concentrate on the study of the commutator
subgroups for discrete graph product groups. Apart from a purely
algebraic interest, our motivation lies in the fact that the
commutator subgroups of graph products are the fundamental groups
of very interesting aspherical spaces. From this topological
perspective, right-angled Coxeter groups $\racg_\sK$ are the most
interesting. The commutator subgroup $\racg_\sK'$ is $\pi_1(\rk)$
for a finite-dimensional aspherical complex $\rk$, which turns out
to be a manifold when $\sK$ is a simplicial subdivision of sphere.
When $\sK$ is a cycle (the boundary of a polygon) or a
triangulated $2$-sphere, one obtains as $\racg_\sK'$ a surface
group or a $3$-manifold group respectively. These groups have
attracted much attention recently in geometric group theory and
low-dimensional topology. The manifolds $\rk$ corresponding to
(the dual complexes of) higher-dimensional \emph{permutahedra} and
\emph{graph-associahedra} also feature as the \emph{universal
realisators} in the works of Gaifullin~\cite{gaif13},~\cite{gaif}
on the problem of realisation of homology classes by manifolds.

In Theorem~\ref{ragws} we give a simple criterion for the
commutator subgroup of a graph product group to be free. In the
case of right-angled Artin groups this result was obtained by
Servatius, Droms and Servatius in~\cite{s-d-s89}. In
Theorem~\ref{gscox} we provide an explicit minimal generator set
for the finitely generated commutator subgroup of a right-angled
Coxeter group~$\racg_\sK$. This generator set consists of nested
iterated commutators of the canonical generators of~$\racg_\sK$
which appear in a special order determined by the combinatorics
of~$\sK$.

Theorems~\ref{ragws} and Theorem~\ref{gscox} parallel the
corresponding results obtained in~\cite{g-p-t-w16} for the loop
homology algebras and rational homotopy Lie algebras of
\emph{moment-angle complexes}. Algebraically, these results
of~\cite{g-p-t-w16} can be interpreted as a description of the
commutator subalgebra in a special graph product graded Lie
algebra (see Theorem~\ref{multgen}). The results of Section~4 in
the current paper constitute a group-theoretic analogue of the
results of~\cite{g-p-t-w16} for graded associative and Lie
algebras.

\medskip

We dedicate this article to the memory of Rainer Vogt, who shared
his great knowledge and ideas with T.\,P. during insightful
collaboration in the 2000s.

The authors are grateful to Alexander Gaifullin for his invaluable
comments and suggestions, which much helped in making the text
more accessible.

\section{Preliminaries}
We consider a finite ordered set $ [m] = \{1,2, \dots, m \} $ and
its subsets $I=\{i_1,\ldots,i_k\}\subset[m]$, where $I$ can be
empty of the whole of~$[m]$.

Let $\sK$ be an (abstract) \emph{simplicial complex} on $[m]$,
i.\,e. $ \mathcal K$ is a collection of subsets of $[m]$ such that
for any $I\in\sK$ all subsets of $I$ also belong to~$\sK$. We
always assume that the empty set $\varnothing$ and all one-element
subsets $\{i\}\subset[m]$ belong to~$\sK$. We refer to $ I \in
\mathcal K $ as a \emph{simplex} (or a \emph{face}) of~$\sK$.
One-element faces are \emph{vertices}, and two-element faces are
\emph{edges}. Every abstract simplicial complex $\sK$ has a
\emph{geometric realisation} $|\sK|$, which is a polyhedron in a
Euclidean space (a union of convex geometric simplices). In all
subsequent constructions it will be useful to keep in mind the
geometric object~$|\sK|$ alongside with the abstract
collection~$\sK$.

We recall the construction of the polyhedral product
(see~\cite{bu-pa00, b-b-c-g10, bu-pa15}).

\begin{construction}[polyhedral product]\label{polpr}
Let $\sK$ be a simplicial complex on~$[m]$ and let
\[
  (\mb X,\mb A)=\{(X_1,A_1),\ldots,(X_m,A_m)\}
\]
be a sequence of $m$ pairs of pointed topological spaces, $\pt\in
A_i\subset X_i$, where $\pt$ denotes the basepoint. For each
subset $I\subset[m]$ we set
\begin{equation}\label{XAI}
  (\mb X,\mb A)^I=\bigl\{(x_1,\ldots,x_m)\in
  \prod_{k=1}^m X_k\colon\; x_k\in A_k\quad\text{for }k\notin I\bigl\}
\end{equation}
and define the \emph{polyhedral product} of $(\mb X,\mb A)$
corresponding to $\sK$ as
\[
  (\mb X,\mb A)^{\sK}=\bigcup_{I\in\mathcal K}(\mb X,\mb A)^I=
  \bigcup_{I\in\mathcal K}
  \Bigl(\prod_{i\in I}X_i\times\prod_{i\notin I}A_i\Bigl).
\]
In the case when all pairs $(X_i,A_i)$ are the same, i.\,e.
$X_i=X$ and $A_i=A$ for $i=1,\ldots,m$, we use the notation
$(X,A)^\sK$ for $(\mb X,\mb A)^\sK$. Also, if each $A_i=\pt$, then
we use the abbreviated notation $\mb X^\sK$ for $(\mb X,\pt)^\sK$,
and $X^\sK$ for $(X,\pt)^\sK$.

This construction of the polyhedral product has the following
categorical interpretation. Consider the \emph{face category}
$\ca(\sK)$, whose objects are simplices $I\in\sK$ and morphisms
are inclusions $I\subset J$. Let $\top$ denote the category of
topological spaces. Define a $\ca{(\sK)}$-diagram (a covariant
functor from the small category $\ca(\sK)$ to the ``large''
category $\top$)
\begin{equation}\label{diagK}
\begin{aligned}
  \mathcal D_\sK(\mb X,\mb A)\colon \ca(\sK)&\longrightarrow \top,\\
  I&\longmapsto (\mb X,\mb A)^I,
\end{aligned}
\end{equation}
which maps the morphism $I\subset J$ of $\ca(\sK)$ to the
inclusion of spaces $(\mb X,\mb A)^I\subset(\mb X,\mb A)^J$. Then
we have
\begin{equation}\label{ppcolim}
  (\mb X,\mb A)^\sK=\mathop{\mathrm{colim}}
  \mathcal D_\sK(\mb X,\mb A)=\mathop{\mathrm{colim}}_{I\in\sK}
  (\mb X,\mb A)^I.
\end{equation}
Here $\colim$ denotes the \emph{colimit functor} (also known as
the \emph{direct limit functor}) from the category of
$\ca{(\sK)}$-diagrams of topological spaces to the
category~$\cat{top}$. By definition, $\colim$ is the left adjoint
to the constant diagram functor. The details of these
constructions can be found, e.\,g., in~\cite[Appendix~C]{bu-pa15}.

\end{construction}

Given a subset~$J\subset[m]$, consider the restriction of $\sK$
to~$J$:
\[
  \sK_J=\{I\in\sK\colon I\subset J\},
\]
which is also known as a \emph{full subcomplex} of~$\sK$. Recall
that a subspace $Y\subset X$ is called a \emph{retract} of~$X$ if
there exists a continuous map $r\colon X\to Y$ such that the
composition $Y\hookrightarrow X\stackrel r\longrightarrow Y$ is
the identity. We record the following simple property of the
polyhedral product.

\begin{proposition}\label{ppret}
$(\mb X,\mb A)^{\sK_J}$ is a retract of $(\mb X,\mb A)^{\sK}$
whenever $\sK_J\subset\sK$ is a full subcomplex.
\end{proposition}
\begin{proof}
We have
\[
  (\mb X,\mb A)^{\sK}=\bigcup_{I\in\mathcal K}
  \Bigl(\prod_{i\in I}X_i\times\prod_{i\in[m]\setminus I}A_i\Bigl),\quad
  (\mb X,\mb A)^{\sK_J}=\bigcup_{I\in\mathcal K,\,I\subset J}
  \Bigl(\prod_{i\in I}X_i\times\prod_{i\in J\setminus I}A_i\Bigl).
\]
Since each $A_i$ is a pointed space, there is a canonical
inclusion $(\mb X,\mb A)^{\sK_J}\hookrightarrow(\mb X,\mb
A)^{\sK}$. Furthermore, for each $I\in\sK$ there is a projection
\[
  r_I\colon \prod_{i\in I}X_i\times\prod_{i\in[m]\setminus
  I}A_i\longrightarrow
  \prod_{i\in I\cap J}X_i\times\prod_{i\in J\setminus I}A_i.
\]
Since $\sK_J$ is a full subcomplex, the image of $r_I$ belongs to
$(\mb X,\mb A)^{\sK_J}$. The projections $r_I$ patch together to
give a retraction $r=\bigcup_{I\in\sK}r_I\colon(\mb X,\mb
A)^\sK\to(\mb X,\mb A)^{\sK_J}$.
\end{proof}

The following examples of polyhedral products feature throughout
the paper.

\begin{example}\label{ppexa}\

\textbf{1.} Let $(X,A)=(S^1,\pt)$, where $S^1$ is a circle. The
corresponding polyhedral product $(S^1)^\sK$ is a subcomplex in
the $m$-torus~$(S^1)^m$:
\begin{equation}\label{s1k}
  (S^1)^\sK=\bigcup_{I\in\sK}(S^1)^I\subset (S^1)^m.
\end{equation}
In particular, when $\sK=\{\varnothing,\{1\},\ldots,\{m\}\}$
(which is $m$ disjoint points geometrically), the polyhedral
product $(S^1)^\sK$ is the wedge $(S^1)^{\vee m}$ of $m$ circles.

When $\sK$ consists of all proper subsets of~$[m]$ (which
geometrically corresponds to the boundary
$\partial\varDelta^{m-1}$ of an $(m-1)$-dimensional simplex),
$(S^1)^\sK$ is known as the \emph{fat wedge} of $m$ circles; it is
obtained by removing the top-dimensional cell from the standard
cell decomposition of an $m$-torus~$(S^1)^m$.

For a general $\sK$ on $m$ vertices, $(S^1)^\sK$ sits between the
$m$-fold wedge $(S^1)^{\vee m}$ and the $m$-fold cartesian
product~$(S^1)^m$.

\smallskip

\textbf{2.} Let $(X,A)=(\R,\Z)$, where $\Z$ is the set of integer
points on a real line~$\R$. We denote the corresponding polyhedral
product by~$\lk$:
\begin{equation}\label{lk}
  \lk=(\R,\Z)^\sK=\bigcup_{I\in\sK}(\R,\Z)^I\subset \R^m.
\end{equation}
When $\sK$ consists of $m$ disjoint points, $\lk$ is a grid in
$m$-dimensional space $\R^m$ consisting of all lines parallel to
one of the coordinate axis and passing though integer points. When
$\sK=\partial\varDelta^{m-1}$, the complex $\lk$ is the union of
all integer hyperplanes parallel to coordinate hyperplanes.

\smallskip

\textbf{3.} Let $(X,A)=(\R P^\infty,\pt)$, where $\R P^\infty$ is
an infinite-dimensional real projective space, which is also the
classifying space $B\Z_2$ for the 2-element cyclic group~$\Z_2$.
Consider the polyhedral product
\begin{equation}\label{rpk}
  (\R P^\infty)^\sK=\bigcup_{I\in\sK}(\R P^\infty)^I\subset (\R P^\infty)^m.
\end{equation}
Similarly to the first example above, $(\R P^\infty)^\sK$ sits
between the $m$-fold wedge $(\R P^\infty)^{\vee m}$ (corresponding
to $\sK$ consisting of $m$ points) and the $m$-fold cartesian
product $(\R P^\infty)^{m}$ (corresponding to
$\sK=\varDelta^{m-1}$).

\smallskip

\textbf{4.} Let $(X,A)=(D^1,S^0)$, where $D^1$ is a closed
interval (a convenient model is the segment $[-1,1]$) and $S^0$ is
its boundary, consisting of two points. The polyhedral product
$(D^1,S^0)^\sK$ is known as the \emph{real moment-angle
complex}~\cite[\S3.5]{bu-pa00},~\cite{bu-pa15} and is denoted
by~$\rk$:
\begin{equation}\label{rk}
  \rk=(D^1,S^0)^\sK=\bigcup_{I\in\sK}(D^1,S^0)^I.
\end{equation}
It is a cubic subcomplex in the $m$-cube $(D^1)^m=[-1,1]^m$. When
$\sK$ consists of $m$ disjoint points, $\rk$ is the 1-dimensional
skeleton of the cube~$[-1,1]^m$. When
$\sK=\partial\varDelta^{m-1}$, $\rk$ is the boundary of the
cube~$[-1,1]^m$. In general, if $\{i_1,\ldots,i_k\}$ is a face of
$\sK$, then $\rk$ contains $2^{m-k}$ cubic faces of dimension $k$
which lie in the $k$-dimensional planes parallel to the
$\{i_1,\ldots,i_k\}$th coordinate plane.

The space $\rk$ was introduced and studied in the works of
Davis~\cite{davi83} and Davis--Januszkiewicz~\cite{da-ja91},
although their construction was different. When $|\sK|$ is
homeomorphic to a sphere, $\rk$ is a topological manifold (this
follows from the results of~\cite{davi83}, see also~\cite{bu-pa15}
and~\cite[Theorem~2.3]{cai}). Furthermore, the manifold $\rk$ has
a smooth structure when $|\sK|$ is the boundary of a convex
polytope. In this case $\rk$ is the \emph{universal abelian cover}
of the dual simple polytope~$P$~\cite[\S4.1]{da-ja91}.
\end{example}

The four polyhedral products above are related by the two homotopy
fibrations~\cite{p-r-v04},~\cite[\S4.3]{bu-pa15}
\begin{gather}
\lk\longrightarrow(S^1)^\sK\longrightarrow(S^1)^m,\label{artfib}\\
\rk\longrightarrow(\R P^\infty)^\sK\longrightarrow(\R
P^\infty)^m.\label{coxfib}
\end{gather}

\begin{construction}[right-angled Artin and Coxeter group]
Let $\Gamma$ be a graph on $m$ vertices. We write
$\{i,j\}\in\Gamma$ when $\{i,j\}$ is an edge. Denote by
$F(g_1,\ldots,g_m)$ a free group with $m$ generators corresponding
to the vertices of~$\Gamma$. The \emph{right-angled Artin group}
$\raag_\Gamma$ corresponding to~$\Gamma$ is defined by generators
and relations as follows:
\begin{equation}\label{raag}
\raag_\Gamma = F(g_1,\ldots,g_m)\big/ (g_ig_j=g_jg_i\text{ for
}\{i,j\}\in\Gamma).
\end{equation}
When $\Gamma$ is a complete graph we have $\raag_\Gamma=\Z^m$,
while when $\Gamma$ has no edges we obtain the free group.

The \emph{right-angled Coxeter group} $\racg_\Gamma$ is defined as
\begin{equation}\label{racg}
\racg_\Gamma = F(g_1,\ldots,g_m)\big/ (g_i^2=1,\;
g_ig_j=g_jg_i\text{ for }\{i,j\}\in\Gamma).
\end{equation}

Both right-angled Artin and Coxeter groups have a categorical
interpretation similar to that of polyhedral products
(see~\eqref{ppcolim}). Namely, consider the following
$\ca{(\sK)}$-diagrams, this time in the category $\cat{grp}$ of
groups:
\[
\begin{aligned}
  \mathcal D_\sK(\Z)\colon
  \ca(\sK)&\longrightarrow\cat{grp},\quad&
  I\longmapsto \Z^I,\\
  \mathcal D_\sK(\Z_2)\colon
  \ca(\sK)&\longrightarrow\cat{grp},\quad&
  I\longmapsto \Z_2^I,
\end{aligned}
\]
where $\Z^I=\prod_{i\in I}\Z$ and $\Z_2^I=\prod_{i\in I}\Z_2$. A
morphism $I\subset J$ of $\ca(\sK)$ is mapped to the monomorphism
of groups $\Z^I\to\Z^J$ and $\Z_2^I\to\Z_2^J$ respectively. Then
\begin{equation}\label{racolimit}
\begin{aligned}
\raag_{\sK^1}=\colim^{\scat{grp}}\mathcal
D_\sK(\Z)=\colim^{\scat{grp}}_{I\in\sK}\Z^I,\\
\racg_{\sK^1}=\colim^{\scat{grp}}\mathcal
D_\sK(\Z_2)=\colim^{\scat{grp}}_{I\in\sK}\Z_2^I,
\end{aligned}
\end{equation}
where $\sK^1$ denotes the 1-skeleton of~$\sK$, which is a graph.
Here $\colim^{\scat{grp}}$ denotes the colimit functor in
$\cat{grp}$.
\end{construction}

A \emph{missing face} (or a \emph{minimal non-face}) of $\sK$ is a
subset $I\subset[m]$ such that $I$ is not a simplex of~$\sK$, but
every proper subset of $I$ is a simplex of~$\sK$. A simplicial
complex $\sK$ is called a \emph{flag complex} if each of its
missing faces consists of two vertices. Equivalently, $\sK$ is
flag if any set of vertices of $\sK$ which are pairwise connected
by edges spans a simplex.

A \emph{clique} (or a \emph{complete subgraph}) of a graph
$\Gamma$ is a subset $I$ of vertices such that every two vertices
in $I$ are connected by an edge. Each flag complex $\sK$ is the
\emph{clique complex} of its one-skeleton $\Gamma=\sK^1$, that is,
the simplicial complex formed by filling in each clique of
$\Gamma$ by a face.

Note that the colimits in~\eqref{racolimit}, being the
corresponding right-angled groups, depend only on the 1-skeleton
of~$\sK$ and do not depend on missing faces with more than 2
vertices. For example, the colimits of the diagrams of groups
$\mathcal D_{\varDelta^2}(\Z)$ and $\mathcal
D_{\partial\varDelta^2}(\Z)$ are both~$\Z^3$. This reflects the
lack of ``higher'' commutativity in the category of groups: when
generators $g_i$ commute pairwise, they commute altogether. This
phenomenon is studied in more detail in~\cite{p-r-v04}
and~\cite{pa-ra08}.

For these reasons we denote the right-angled Artin and Coxeter
groups corresponding to the 1-skeleton of $\sK$ simply by
$\raag_\sK$ and $\racg_\sK$ respectively.

By analogy with the polyhedral product of spaces $\mb
X^\sK=\colim_{I\in\sK}\mb X^I$, we may consider the following more
general construction of a discrete group.

\begin{construction}[graph product]\label{cgrpr}
Let $\sK$ be a simplicial complex on~$[m]$ and let $\mb
G=(G_1,\ldots,G_m)$ be a sequence of $m$ groups, which we think of
as discrete topological groups. We also assume that none of $G_i$
is trivial, i.\,e. $G_i\ne\{1\}$. For each subset $I\subset[m]$ we
set
\[
  \mb G^I=\bigl\{(g_1,\ldots,g_m)\in
  \prod_{k=1}^m G_k\colon\; g_k=1\quad\text{for }k\notin I\bigl\}.
\]
Then consider the following $\ca{(\sK)}$-diagram of groups:
\[
\begin{aligned}
  \mathcal D_\sK(\mb G)\colon
  \ca(\sK)&\longrightarrow\cat{grp},\quad&
  I\longmapsto \mb G^I,
\end{aligned}
\]
which maps a morphism $I\subset J$ to the canonical monomorphism
of groups $\mb G^I\to\mb G^J$. Define the group
\begin{equation}\label{ragcolimit}
  \mb G^{\sK}=\colim^{\scat{grp}}\mathcal D_\sK(\mb
  G)=\colim^{\scat{grp}}_{I\in\sK}\mb G^I.
\end{equation}
The group $\mb G^\sK$ depends only on the graph $\sK^1$ and is
called the \emph{graph product} of the groups $G_1,\ldots,G_m$. We
have canonical homomorphisms $\mb G^I\to\mb G^\sK$, $I\in\sK$,
which can be shown to be injective.
\end{construction}

As in the case of right-angled Artin and Coxeter groups
(corresponding to $G_i=\Z$ and $G_i=\Z_2$ respectively), one
readily deduces the following more explicit description from the
universal property of the colimit:

\begin{proposition}
The is an isomorphism of groups
\[
  \mb G^{\sK}\cong\mathop{\mbox{\Huge$\star$}}_{k=1}^m G_k\big/(g_ig_j=g_jg_i\,\text{ for
}g_i\in G_i,\,g_j\in G_j,\,\{i,j\}\in\sK),
\]
where $\mathop{\mbox{\Huge$\star$}}_{k=1}^m G_k$ denotes the free
product of the groups~$G_k$.
\end{proposition}

\begin{remark}
We use the symbol $\star$ to denote the free product of groups,
instead of the more common~$\ast$; the latter is reserved for the
join of topological spaces.
\end{remark}

\section{Classifying spaces}
Here we collect the information about the classifying spaces for
graph product groups. The results of this section are not new, but
as they are spread across the literature we find it convenient to
collect everything in one place. The corresponding references are
given below.

Recall that a path-connected space $X$ is \emph{aspherical} if
$\pi_i(X)=0$ for $i\ge2$. An aspherical space $X$ is an
Eilenberg--Mac Lane space $K(\pi,1)$ with $\pi=\pi_1(X)$.

Given a (discrete) group $G$, there is a \emph{universal
$G$-covering} $EG\to BG$ whose total space $EG$ is contractible
and the base $BG$, known as the \emph{classifying space} for~$G$,
has the homotopy type~$K(G,1)$ (i.\,e. $\pi_1(BG)=G$ and
$\pi_i(BG)=0$ for $i\ge2$). We shall therefore switch between the
notation $BG$ and $K(G,1)$ freely.

Note that $B\Z\simeq S^1$ and $B\Z_2\simeq\R P^\infty$, with the
universal coverings $\R\to S^1$ and $S^\infty\to\R P^\infty$
respectively.

Now we use the notation from Construction~\ref{cgrpr}. The
classifying space $B\mb G^I$ is the product of $BG_i$ over $i\in
I$. We therefore have the polyhedral product $(B\mb G)^\sK$
corresponding to the sequence of pairs $(B\mb
G,\pt)=\{(BG_1,\pt),\ldots,(BG_m,\pt)\}$.
Similarly, we have the polyhedral product $(E\mb G,\mb G)^\sK$
corresponding to the sequence of pairs $(E\mb G,\mb
G)=\{(EG_1,G_1),\ldots,(EG_m,G_m)\}$. Here each $G_i$ is included
in $EG_i$ as the fibre of the covering $EG_i\to BG_i$ over the
basepoint.

The homotopy fibrations~\eqref{artfib} and~\eqref{coxfib} can be
generalised as follows.

\begin{proposition}\label{ragfib}
The sequence of canonical maps
\[
  (E\mb G,\mb G)^\sK\longrightarrow (B\mb
  G)^\sK\longrightarrow\prod_{k=1}^m BG_k
\]
is a homotopy fibration.
\end{proposition}

When each $G_k$ is $\Z$, we obtain the fibration~\eqref{artfib},
as the pair $(E\Z,\Z)$ is homotopy equivalent to $(\R,\Z)$.
Similarly, when each $G_k$ is $\Z_2$, we obtain~\eqref{coxfib}, as
the pair $(E\Z_2,\Z_2)$ is homotopy equivalent to $(D^1,S^0)$.

\begin{proof}[Proof of Proposition~\ref{ragfib}]
We denote $\prod_{k=1}^m BG_k$ by $B\mb G^{[m]}$; this is
compatible with the notation $B\mb G^I$. According
to~\cite[Proposition~5.1]{p-r-v04}, the homotopy fibre of the
inclusion $(B\mb G)^\sK\to B\mb G^{[m]}$ can be identified with
the homotopy colimit $\hocolim_{I\in\sK}\mb G^{[m]}/\mb G^I$ of
the $\ca(\sK)$-diagram in $\cat{top}$ given on the objects by
$I\mapsto\mb G^{[m]}/\mb G^I$ (where the latter is the quotient
group, viewed as a discrete space) and sending a morphism
$I\subset J$ to the canonical projection $\mb G^{[m]}/\mb G^I\to
\mb G^{[m]}/\mb G^J$ of the quotients. This diagram is not Reedy
cofibrant, e.\,g. because $\mb G^{[m]}/\mb G^I\to\mb G^{[m]}/\mb
G^J$ is not a cofibration of spaces. The latter map is homotopy
equivalent to the closed cofibration $(E\mb G,\mb G)^I\to(E\mb
G,\mb G)^J$, which is a morphism in the $\ca(\sK)$-diagram
$\mathcal D_\sK(E\mb G,\mb G)$, see~\eqref{diagK}. The diagram
$\mathcal D_\sK(E\mb G,\mb G)$ is Reedy cofibrant,
see~\cite[Proposition~8.1.1]{bu-pa15}. Therefore, the homotopy
fibre of the inclusion $(B\mb G)^\sK\to B\mb G^{[m]}$ is given by
\[
  \hocolim_{I\in\sK}\mb G^{[m]}/\mb G^I\simeq
  \colim_{I\in\sK}(E\mb G,\mb G)^I=(E\mb G,\mb G)^\sK.\qedhere
\]
\end{proof}


Now we state the following group-theoretic consequence of the
homotopy fibration in Proposition~\ref{ragfib}.

\begin{theorem}\label{gpfund}
Let $\sK$ be a simplicial complex on $m$ vertices, and let $\mb
G^{\sK}$ be a graph product group given by~\eqref{ragcolimit}.
\begin{itemize}
\item[(a)] $\pi_1((B\mb G)^\sK)\cong\mb G^\sK$.
\item[(b)] Both spaces $(B\mb G)^\sK$ and $(E\mb G,\mb G)^\sK$ are aspherical if and only if $\sK$ is
flag. Hence, $B(\mb G^\sK)=(B\mb G)^\sK$ whenever $\sK$ is flag.
\item[(c)] $\pi_i((B\mb G)^\sK)\cong\pi_i((E\mb G,\mb G)^\sK)$ for $i\ge2$.
\item[(d)] $\pi_1((E\mb G,\mb G)^\sK)$ is isomorphic to the kernel of the canonical projection
$\mb G^\sK\to\prod_{k=1}^m G_k$.
\end{itemize}
\end{theorem}

\begin{proof}
To prove (a) we proceed inductively by adding simplices to~$\sK$
one by one and use van Kampen's Theorem. The base of the induction
is $\sK$ consisting of $m$ disjoint points. Then $(B\mb G)^\sK$ is
the wedge $BG_1\vee\cdots\vee BG_m$, and $\pi_1((B\mb G)^\sK)$ is
the free product $G_1\mathbin\star\cdots\mathbin\star G_m$. This
is precisely $\mb G^\sK$, so (a) holds. Assume now that $\sK'$ is
obtained from $\sK$ by adding a single 1-dimensional
simplex~$\{i,j\}$. Then, by the definition of the polyhedral
product,
\[
  (B\mb G)^{\sK'}=(B\mb G)^\sK\cup (BG_i\times BG_j),
\]
where the two pieces are glued along $BG_i\vee BG_j$. By van
Kampen's Theorem, $\pi_1((B\mb G)^{\sK'})$ is the amalgamated free
product $\pi_1((B\mb G)^\sK)\mathbin\star_{(G_i\mathbin\star
G_j)}(G_i\times G_j)$. The latter group is obtained from
$\pi_1((B\mb G)^\sK)$ by adding all relations of the form
$g_ig_j=g_jg_i$ for $g_i\in G_i$, $g_j\in G_j$. By the inductive
assumption, this is precisely $\mb G^{\sK'}$. Adding simplices of
dimension $\ge2$ to $\sK$ does not change $\mb G^\sK$ and results
in adding cells of dimension $\ge3$ to $(B\mb G)^\sK$, which does
not change the fundamental group $\pi_1((B\mb G)^\sK)$. The
inductive step is therefore complete, proving~(a).

Now we prove (b). The canonical homomorphisms $\mb G^I\to\mb
G^\sK$ give rise to the maps of classifying spaces $B\mb G^I\to
B(\mb G^\sK)$. These define a morphism from the $\ca(\sK)$-diagram
$\mathcal D_\sK(B\mb G,\pt)$ to the constant diagram $B(\mb
G^\sK)$, and hence a map
\begin{equation}\label{bcolim}
  \colim_{I\in\sK} B\mb G^I=(B\mb G)^\sK\to B(\mb G^\sK).
\end{equation}
According to~\cite[Proposition~5.1]{p-r-v04}, the homotopy fibre
of the map~\eqref{bcolim} can be identified with the homotopy
colimit $\hocolim_{I\in\sK}\mb G^\sK/\mb G^I$ of the
$\ca(\sK)$-diagram in $\cat{top}$ given on the objects by
$I\mapsto\mb G^\sK/\mb G^I$ (where the latter is the right coset,
viewed as a discrete space) and sending a morphism $I\subset J$ to
the canonical projection $\mb G^\sK/\mb G^I\to\mb G^\sK/\mb G^J$
of cosets. By~\cite[Corollary~5.4]{p-r-v04}, the homotopy  colimit
$\hocolim_{I\in\sK}\mb G^\sK/\mb G^I$ is homeomorphic to the
identification space
\begin{equation}\label{idspace}
  \bigl(B\ca(\sK)\times\mb G^\sK\bigr)\big/\!\sim.
\end{equation}
Here $B\ca(\sK)$ is the classifying space of $\ca(\sK)$, which is
homeomorphic to the cone on~$|\sK|$. The equivalence relation
$\sim$ is defined as follows: $(x,gh)\mathop\sim(x,g)$ whenever
$h\in\mb G^I$ and $x\in B(I\mathop\downarrow\ca(\sK))$, where
$I\mathop\downarrow\ca(\sK)$ is the \emph{undercategory}, whose
objects are $J\in\sK$ such that $I\subset J$, and
$B(I\mathop\downarrow\ca(\sK))$ is homeomorphic to the star of $I$
in~$\sK$. When $\sK$ is a flag complex, the identification
space~\eqref{idspace} is contractible
by~\cite[Proposition~6.1]{p-r-v04}. Therefore, the
map~\eqref{bcolim} is a homotopy equivalence, which implies that
$(B\mb G)^\sK$ is aspherical when $\sK$ is flag.

Assume now that $\sK$ is not flag. Choose a missing face
$J=\{j_1,\ldots,j_k\}\subset[m]$ with $k\ge3$ vertices and
consider the corresponding full subcomplex $\sK_{J}$. Then $(B\mb
G)^{\sK_J}$ is the fat wedge of the spaces $\{BG_j,j\in J\}$ (see
Example~\ref{ppexa}.1), and it is a retract of $(B\mb G)^{\sK}$ by
Proposition~\ref{ppret}. Hence, in order to see that $(B\mb
G)^{\sK}$ is not aspherical, it is enough to check that $(B\mb
G)^{\sK_J}$ is not aspherical. Let $\fw(X_1,\ldots,X_k)$ denote
the fat wedge of spaces $X_1,\ldots,X_k$. According to a result of
Porter~\cite{port66}, the homotopy fibre of the inclusion
\[
  \fw(X_1,\ldots,X_k)\hookrightarrow\prod_{i=1}^k X_i
\]
is $\varSigma^{k-1}\varOmega X_1\wedge\cdots\wedge\varOmega X_k$,
where $\varSigma$ denotes the suspension and $\varOmega$ the loop
space functor. In our case we obtain that the homotopy fibre of
the inclusion $(B\mb G)^{\sK_J}\to\prod_{j\in J} BG_j$ is
$\varSigma^{k-1}G_{j_1}\wedge\cdots\wedge G_{j_k}$. Since each
$G_j$ is a discrete space, the latter suspension is a wedge of
$(k-1)$-dimensional spheres. It has nontrivial homotopy group
$\pi_{k-1}$. Since $\prod_{j\in J}BG_j$ is a $K(\pi,1)$-space, the
homotopy exact sequence implies that $\pi_{k-1}((B\mb
G)^{\sK_J})\ne0$ for some $k\ge3$. Hence, $(B\mb G)^{\sK_J}$ and
$(B\mb G)^{\sK}$ are non-aspherical.

Asphericity of~$(E\mb G,\mb G)^\sK$ and statements~(c) and~(d)
follow from the homotopy exact sequence of the fibration in
Proposition~\ref{ragfib}, as $\pi_i(\prod_{k=1}^m BG_k)=0$,
$i\ge2$.
\end{proof}

Specialising to the cases $G_k=\Z$ and $G_k=\Z_2$ respectively we
obtain the following results about right-angled Artin and Coxeter
groups. Note that in these two cases the groups $G_k$ are abelian,
so $\mb G^\sK\to\prod_{k=1}^m G_k$ is the abelianisation
homomorphism, and its kernel is the commutator subgroup~$(\mb
G^\sK){\vphantom{\bigr)}}'$.

\begin{corollary}\label{artfund}
Let $\sK$ be a simplicial complex on $m$ vertices, let $(S^1)^\sK$
and $\lk$ be the polyhedral products given by~\eqref{s1k}
and~\eqref{lk} respectively, and let $\raag_\sK$ be the
corresponding right-angled Artin group.
\begin{itemize}
\item[(a)] $\pi_1((S^1)^\sK)\cong\raag_\sK$.
\item[(b)] Both $(S^1)^\sK$ and $\lk$ are aspherical if and only if $\sK$ is
flag.
\item[(c)] $\pi_i((S^1)^\sK)\cong\pi_i(\lk)$ for $i\ge2$.
\item[(d)] $\pi_1(\lk)$ is isomorphic to the commutator subgroup~$\raag'_\sK$.
\end{itemize}
\end{corollary}

\begin{corollary}\label{coxfund}
Let $\sK$ be a simplicial complex on $m$ vertices, let $(\R
P^\infty)^\sK$ and $\rk$ be the polyhedral products given
by~\eqref{rpk} and~\eqref{rk} respectively, and let $\racg_\sK$ be
the corresponding right-angled Coxeter group.
\begin{itemize}
\item[(a)] $\pi_1((\R P^\infty)^\sK)\cong\racg_\sK$.
\item[(b)] Both $(\R P^\infty)^\sK$  and $\rk$ are aspherical if and only if $\sK$ is
flag.
\item[(c)] $\pi_i((\R P^\infty)^\sK)\cong\pi_i(\rk)$ for $i\ge2$.
\item[(d)] $\pi_1(\rk)$ is isomorphic to the commutator subgroup~$\racg'_\sK$.
\end{itemize}
\end{corollary}

\begin{remark}
All ingredients in the proof of Theorem~\ref{gpfund} are contained
in~\cite{p-r-v04}. The fact that the polyhedral product $(B\mb
G)^\sK$ is the classifying space for the graph product group $\mb
G^\sK$ whenever $\sK$ is a flag complex implies that the
classifying space functor converts the colimit of groups (defining
the graph product) to the colimit of topological spaces (defining
the polyhedral product). This is not the case when $\sK$ is not
flag because of the presence of higher Whitehead and Samelson
products (see~\cite{p-r-v04,pa-ra08,gr-th16}), but the situation
can be remedied by replacing colimits with homotopy colimits. All
these facts were proved in~\cite{p-r-v04} for arbitrary
well-pointed topological groups.

Statements (a) and (b) of Corollary~\ref{artfund}, implying a
homotopy equivalence $(S^1)^\sK\simeq K(\raag_\sK,1)$ for
flag~$\sK$, were obtained by Kim and
Roush~\cite[Theorem~10]{ki-ro80}. Statements (a) and (b) of
Corollary~\ref{coxfund}, implying a homotopy equivalence $(\R
P^\infty)^\sK\simeq K(\racg_\sK,1)$ for flag~$\sK$, are implicit
in the works of Davis~\cite{davi83} and
Davis--Januszkiewicz~\cite[p.~437]{da-ja91}. In particular,
contractibility of the space~\eqref{idspace} (which is the crucial
step in the proof of  Theorem~\ref{gpfund}~(b)) in the case of
right-angled Coxeter group $\racg_\sK$ follows
from~\cite[Theorem~13.5]{davi83}. The isomorphism between
$\pi_1(\rk)$ and the commutator subgroup $\racg'_\sK$ was also
obtained in the work of Droms~\cite{drom03} (his cubic complex is
the 2-dimensional skeleton of our complex~$\rk$, and therefore has
the same fundamental group).

In the case of a general graph product $\mb G^\sK$, the result
that both spaces $(B\mb G)^\sK$ and $(E\mb G,\mb G)^\sK$ are
aspherical if and only if $\sK$ is flag appeared in the work of
Stafa~\cite[Theorem~1.1]{staf15}.
\end{remark}

\begin{example}\label{mcyc1}
Let $\sK$ be an $m$-cycle (the boundary of an $m$-gon). A simple
argument with Euler characteristic shows that $\rk$ is
homeomorphic to a closed orientable surface of genus
$(m-4)2^{m-3}+1$ (this observation goes back to a 1938 work of
Coxeter, see~\cite[Proposition~4.1.8]{bu-pa15}). Therefore, the
commutator subgroup of the corresponding right-angled Coxeter
group~$\racg_\sK$ is a surface group. This example was studied
in~\cite{s-d-s89} and~\cite{drom03}.

Similarly, when $|\sK|\cong S^2$ (which is equivalent to $\sK$
being the boundary of a 3-dimensional simplicial polytope), $\rk$
is a 3-dimensional manifold. Therefore, the commutator subgroup of
the corresponding $\racg_\sK$ is a $3$-manifold group. The fact
that 3-manifold groups appear as subgroups in right-angled Artin
and Coxeter groups has attracted much attention in the recent
literature.
\end{example}

All homology groups are considered with integer coefficients.  The
homology of $\rk$ is described by the following result. For the
particular case of flag $\sK$ it gives a description of the
homology of the commutator subgroup~$\racg'_\sK$.

\begin{theorem}[{\cite{bu-pa00},
\cite[\S4.5]{bu-pa15}}]\label{homrk}
For any $k\ge0$, there is an isomorphism
\[
  H_k(\rk)\cong\bigoplus_{J\subset[m]}\widetilde
  H_{k-1}(\sK_J),
\]
where $\widetilde H_{k-1}(\sK_J)$ is the reduced simplicial
homology group of~$\sK_J$.
\end{theorem}

%
The cohomology ring structure of $H^*(\rk)$ is described
in~\cite{cai}.

\section{The structure of the commutator subgroups}
By Theorem~\ref{gpfund},
\[
  \Ker\Bigl(\mb G^\sK\to\prod_{k=1}^m G_k\Bigr)=\pi_1((E\mb G,\mb
  G)^\sK).
\]
In the case of right-angled Artin or Coxeter groups (or, more
generally, when each $G_k$ is abelian), the group above is the
commutator subgroup $(\mb G^\sK){\vphantom{\bigr)}}'$. We want to
study the group $\pi_1((E\mb G,\mb  G)^\sK)$, identify the class
of simplicial complexes $\sK$ for which this group is free, and
describe a minimal generator set.

We shall need the following modification of a result of Grbi\'c
and Theriault~\cite{gr-th16}:

\begin{proposition}\label{attach}
Let $\sK=\sK_1\cup_I\sK_2$ be a simplicial complex obtained by
gluing $\sK_1$ and $\sK_2$ along a common face~$I$, which may be
empty. If the polyhedral products $(E\mb G,\mb G)^{\sK_1}$ and
$(E\mb G,\mb G)^{\sK_2}$ are homotopy equivalent to wedges of
circles, then $(E\mb G,\mb G)^{\sK}$ is also homotopy equivalent
to a wedge of circles.
\end{proposition}
\begin{proof}
We may assume that $\sK$ has the vertex set~$[m]=\{1,\ldots,m\}$,
$\sK_1$ is the full subcomplex of $\sK$ on the first $m_1$
vertices $\{1,\ldots,m_1\}$, $\sK_2$ is the full subcomplex of
$\sK$ on the last $m_2$ vertices $\{m-m_2+1,\ldots,m\}$, and the
common face $I$ is on the $k$ vertices $\{m_1-k+1,\ldots,m_1\}$,
where $m_1<m$, $m_2<m$ and $m=m_1+m_2-k$. Consider  the polyhedral
product $(C\mb X,\mb X)^\sK$ corresponding to a sequence of pairs
$(C\mb X,\mb X)=\{(CX_1,X_1),\ldots,(CX_m,X_m)\}$, where $CX_i$
denotes the cone on~$X_i$. According
to~\cite[Theorem~6.12]{gr-th16},
\begin{equation}\label{gtdec}
  (C\mb X,\mb X)^\sK\simeq(M_1\mathbin\ast M_2)\vee
  ((C\mb X,\mb X)^{\sK_1}\rtimes M_2)\vee
  (M_1\ltimes (C\mb X,\mb X)^{\sK_2}),
\end{equation}
where $M_1=\prod_{i=1}^{m_1}X_i$, $M_2=\prod_{i=m-m_2+1}^m X_i$,
$M_1\mathbin\ast M_2$ denotes the \emph{join} of $M_1$ and $M_2$,
$X\rtimes Y$  denotes the \emph{right half-smash} $X\times
Y/\pt\times Y$ of two pointed spaces $X,Y$, and $X\ltimes Y$
denotes their \emph{left half-smash} $X\times Y/X\times\pt$.

In our case, each $X_i=G_i$ is a discrete space, the pair
$(EG_i,G_i)$ is homotopy equivalent to $(CG_i,G_i)$, and each of
$M_1,M_2$ in~\eqref{gtdec} is a discrete space. Hence, each of the
three wedge summands in~\eqref{gtdec} is a wedge of circles, and
so is $(E\mb G,\mb G)^{\sK}$.
\end{proof}

A graph $\Gamma$ is called \emph{chordal} (or \emph{triangulated})
if each of its cycles with $\ge 4$ vertices has a chord (an edge
joining two vertices that are not adjacent in the cycle).

The following result gives an alternative characterisation of
chordal graphs.

\begin{theorem}[Fulkerson--Gross~\cite{fu-gr65}]
A graph is chordal if and only if its vertices can be ordered in
such a way that, for each vertex~$i$, the lesser neighbours of~$i$
form a clique.
\end{theorem}

Such an ordering of vertices is called a \emph{perfect elimination
ordering}.

\begin{theorem}\label{ragws}
Let $\sK$ be a flag simplicial complex on~$m$ vertices, let $\mb
G=(G_1,\ldots,G_m)$ be a sequence of $m$ nontrivial groups, and
let $\mb G^{\sK}$ be the graph product group given
by~\eqref{ragcolimit}. The following conditions are equivalent:
\begin{itemize}
\item[(a)] $\Ker(\mb G^\sK\to\prod_{k=1}^m G_k)$ is a free group;
\item[(b)] $(E\mb G,\mb G)^\sK$ is homotopy equivalent to a wedge of circles;
\item[(c)] $\Gamma=\sK^1$ is a chordal graph.
\end{itemize}
\end{theorem}

\begin{proof}
(b)$\Rightarrow$(a) This follows from Theorem~\ref{gpfund}~(d) and
the fact that the fundamental group of a wedge of circles is
free.

\smallskip

(c)$\Rightarrow$(b) Here we use the argument
from~\cite[Theorem~4.6]{g-p-t-w16}. However, that argument
contained an inaccuracy, which was pointed out by A.~Gaifullin and
corrected in the argument below.

Assume that the vertices of $\mathcal K$ are in perfect
elimination order. We assign to each vertex $i$ the clique $I_i$
consisting of $i$ and the lesser neighbours of $i$. Since $\sK$ is
a flag complex, each clique $I_i$ is a face. All maximal faces are
among $I_1,\ldots,I_m$, so we have $\bigcup_{i=1}^m I_i=\sK$.
Furthermore, for each $k=1,\ldots,m$ the perfect elimination
ordering on $\sK$ induces such an ordering on the full subcomplex
$\sK_{\{1,\ldots,k-1\}}$, so we have
$\bigcup_{i=1}^{k-1}I_i=\sK_{\{1,\ldots,k-1\}}$. In particular,
the simplicial complex $\bigcup_{i=1}^{k-1}I_i$ is flag as a full
subcomplex in a flag complex. The intersection
$I_k\cap\bigcup_{i=1}^{k-1}I_i$ is a clique, so it is a face of
$\bigcup_{i=1}^{k-1}I_i$. An inductive argument using
Proposition~\ref{attach} then shows that $(E\mb G,\mb G)^\sK$ is a
wedge of circles.

\smallskip

(a)$\Rightarrow$(c) Let $\Ker(\mb G^\sK\to\prod_{k=1}^m G_k)$ be a
free group. Suppose that the graph $\Gamma=\sK^1$ is not chordal,
and choose a chordless cycle $J$ with $|J|\ge4$. Then the full
subcomplex $\sK_J$ is the same cycle (the boundary of
a~$|J|$-gon).

We first consider the case when each $G_k$ is $\Z_2$, so that
$(E\mb G,\mb G)^\sK$ is $\rk$. Then $\mathcal R_{\mathcal K_J}$ is
homeomorphic to a closed orientable surface of genus
$(|J|-4)2^{|J|-3}+1$ by~\cite[Proposition~4.1.8]{bu-pa15}. In
particular, the fundamental group $\pi_1(\mathcal R_{\mathcal
K_J})$ is not free. On the other hand, $\mathcal R_{\mathcal K_J}$
is a retract of $\rk$ by Proposition~\ref{ppret}, so
$\pi_1(\mathcal R_{\mathcal K_J})$ is a subgroup of the free group
$\pi_1(\rk)=\Ker(\racg_\sK\to(\Z_2)^m)$. A contradiction.

Now consider the general case. Note that the pair $(EG_k,G_k)$ is
homotopy equivalent to $(CG_k,G_k)$, so we can consider $(C\mb
G,\mb G)^\sK$ instead of $(E\mb G,\mb G)^\sK$. Since each $G_k$ is
discrete and nontrivial, we may fix an inclusion of a pair of
points $S^0\hookrightarrow G_k$; then there is a retraction
$G_k\to S^0$ (it does not have to be a homomorphism of groups). It
extends to a retraction of cones, so we have a retraction of pairs
$(CG_k,G_k)\to(D^1,S^0)$. These retractions give rise to a
retraction of polyhedral products $(C\mb G,\mb
G)^\sK\to(D^1,S^0)^\sK=\rk$. Hence, we have a composite retraction
$(C\mb G,\mb G)^\sK\to\rk\to\mathcal R_{\mathcal K_J}$, so
$\pi_1(\mathcal R_{\mathcal K_J})$ includes as a subgroup in the
free group $\pi_1(E\mb G,\mb G)^\sK=\Ker(\mb G^\sK\to\prod_{k=1}^m
G_k)$. On the other hand, if $\sK^1$ contains a chordless cycle
$J$ with $|J|\ge4$, then $\pi_1(\mathcal R_{\mathcal K_J})$ is the
fundamental group of a surface of positive genus, so it is not
free. A contradiction.
\end{proof}

\begin{corollary}\label{commrag}
Let $\raag_\sK$ and $\racg_\sK$ be the right-angled Artin and
Coxeter groups corresponding to a simplicial complex~$\sK$.
\begin{itemize}
\item[(a)] The commutator subgroup $\raag'_\sK$ is free if
and only if $\sK^1$ is a chordal graph.
\item[(b)] The commutator subgroup $\racg'_\sK$ is free if
and only if $\sK^1$ is a chordal graph.
\end{itemize}
\end{corollary}

Part (a) of Corollary~\ref{commrag} is the result of Servatius,
Droms and Servatius~\cite{s-d-s89}. The difference between parts
(a) and (b) is that the commutator subgroup $\raag'_\sK$ is
infinitely generated, unless $\raag_\sK=\Z^m$, while the
commutator subgroup $\racg'_\sK$ is finitely generated. We
elaborate on this in the next theorem.

Let $(g,h)=g^{-1}h^{-1}gh$ denote the group commutator of two
elements $g,h$.

\begin{theorem}\label{gscox}
Let $\racg_\sK$ be the right-angled Coxeter group corresponding to
a simplicial complex~$\sK$ on $m$ vertices. The commutator
subgroup $\racg'_\sK$ has a finite minimal generator set
consisting of $\sum_{J\subset[m]}\rank\widetilde H_0(\sK_J)$
iterated commutators
\begin{equation}\label{commuset}
  (g_j,g_i),\quad (g_{k_1},(g_j,g_i)),\quad\ldots,\quad
  (g_{k_1},(g_{k_2},\cdots(g_{k_{m-2}},(g_j,g_i))\cdots)),
\end{equation}
where $k_1<k_2<\cdots<k_{\ell-2}<j>i$, $k_s\ne i$ for any~$s$,
and $i$ is the smallest vertex in a connected component not
containing~$j$ of the subcomplex
$\sK_{\{k_1,\ldots,k_{\ell-2},j,i\}}$.
\end{theorem}

Theorem~\ref{gscox} is similar to a result of~\cite{g-p-t-w16}
describing the commutator subalgebra of the graded Lie algebra
given by
\begin{equation}\label{liek}
  L_\sK=\mathop{\mbox{\textit{FL}}}\langle u_{1},\ldots,u_{m}\rangle
  \big/\bigl([u_i,u_i]=0,\; [u_i,u_j]=0\text{ for }\{i,j\}\in
  \sK\bigr),
\end{equation}
where $\mathop{\mbox{\textit{FL}}}\langle
u_{1},\ldots,u_{m}\rangle$ is the free graded Lie algebra on
generators $u_i$ of degree one, and $[a,b]=-(-1)^{|a||b|}[b,a]$
denotes the graded Lie bracket. The commutator subalgebra is the
kernel of the Lie algebra homomorphism
$L_\sK\to\mathop{\mbox{\textit{CL}}}\langle
u_{1},\ldots,u_{m}\rangle$ to the commutative (trivial) Lie
algebra.

The graded Lie algebra~\eqref{liek} is a graph product similar to
the right-angled Coxeter group $\racg_\sK$. It has a colimit
decomposition similar to~\eqref{ragcolimit}, with each $G_i$
replaced by the trivial Lie algebra
$\mathop{\mbox{\textit{CL}}}\langle
u\rangle=\mathop{\mbox{\textit{FL}}}\langle u\rangle/([u,u]=0)$
and the colimit taken in the category of graded Lie algebras.

\begin{theorem}[{\cite[Theorem~4.3]{g-p-t-w16}}]\label{multgen}
The commutator subalgebra of the graded Lie algebra $L_\sK$ has a
finite minimal generator set consisting of
$\sum_{J\subset[m]}\rank\widetilde H_0(\sK_J)$ iterated
commutators
\[
  [u_j,u_i],\quad [u_{k_1},[u_j,u_i]],\quad\ldots,\quad
  [u_{k_1},[u_{k_2},\cdots[u_{k_{m-2}},[u_j,u_i]]\cdots]],
\]
where $k_1<k_2<\cdots<k_{\ell-2}<j>i$, $k_s\ne i$ for any~$s$, and
$i$ is the smallest vertex in a connected component not
containing~$j$ of the subcomplex
$\sK_{\{k_1,\ldots,k_{\ell-2},j,i\}}$.
\end{theorem}

Although the scheme of the proof of Theorem~\ref{gscox} is similar
to that for Theorem~\ref{multgen}, more specific techniques are
required to work with group commutators, as opposed to Lie algebra
brackets. Nevertheless, most of these techniques are quite
standard, and can be extracted from the classical texts
like~\cite{m-k-s76}.

\begin{proof}[Proof of Theorem~\ref{gscox}]
The first part is a standard argument applicable to the commutator
subgroup of an arbitrary group. An element of $\racg'_\sK$ is a
product of commutators $(a,b)$ with $a,b\in\racg_\sK$. Writing
each of $a,b$ as a word in the generators $g_1,\ldots,g_m$ and
using the Hall identities
\begin{equation}\label{hall}
\begin{aligned}
(a,bc)&=(a,c)(a,b)((a,b),c),\\
(ab,c)&=(a,c)((a,c),b)(b,c),
\end{aligned}
\end{equation}
we express each element of $\racg'_\sK$ in terms of iterated
commutators $(g^{n_{i_1}}_{i_1},\ldots,g^{n_{i_\ell}}_{i_\ell})$ with $n_{i_k}\in\Z$ and arbitrary
bracketing. Since we have relations $g_i^2=1$ in $\racg_\sK$, we may assume that each $n_{i_k}$ is~$1$. We refer to $\ell\ge2$ as the \emph{length} of an
iterated commutator. If an iterated commutator
$(g_{i_1},\ldots,g_{i_\ell})$ contains a commutator $(a,b)$ where
each of $a,b$ is itself a commutator, then we can remove such
$(g_{i_1},\ldots,g_{i_\ell})$ from the list of generators by
writing $(a,b)$ as a word in shorter commutators $a,b$ and
using~\eqref{hall} iteratively. We therefore obtain a generators
set for $\racg'_\sK$ consisting only of \emph{nested} iterated
commutators, i.\,e. those not containing $(a,b)$ where both $a,b$
are commutators. The next step is to use the identity
\[
  ((a,b),c)=(b,a)(c,(b,a))(a,b)
\]
and the identities~\eqref{hall} to express each nested commutator
in terms of \emph{canonical} nested commutators
$(g_{i_1},(g_{i_2},\cdots(g_{i_{\ell-2}},(g_{i_{\ell-1}},g_{i_\ell}))\cdots))$.

The most important part of the proof is to express each canonical
nested commutator in terms of canonical nested commutators in
which the generators $g_i$ appear in a specific order. This will
be done by a combination of algebraic and topological arguments
and use the specifics of the group~$\racg_\sK$.

We first prove a particular case of the statement, corresponding
to $\sK$ consisting of $m$ disjoint points. The group $\racg_\sK$
is then a free product of $m$ copies of~$\Z_2$.

\begin{lemma}\label{fpgen}
Let $G$ be a free product of $m$ copies of~$\Z_2$, given by the
presentation
\[
  G= F(g_1,\ldots,g_m)\big/ (g_i^2=1,\quad i=1,\ldots,m).
\]
Then the commutator subgroup $G'$ is a free group freely generated
by the iterated commutators of the form
\[
  (g_j,g_i),\quad (g_{k_1},(g_j,g_i)),\quad\ldots,\quad
  (g_{k_1},(g_{k_2},\cdots(g_{k_{m-2}},(g_j,g_i))\cdots)),
\]
where $k_1<k_2<\cdots<k_{\ell-2}<j>i$ and $k_s\ne i$ for any~$s$.
Here, the number of commutators of length $\ell$ is equal to
$(\ell-1)\binom m\ell$.
\end{lemma}
\begin{proof}
The statement is clear for $m=1$ (then $G=\Z_2$) and for $m=2$
(then $G=\Z_2\mathbin{\star}\Z_2$ and $G'\cong\Z$ with generator
$(g_2,g_1)$). For $m=3$, the lemma says that the commutator
subgroup of $G=\Z_2\mathbin{\star}\Z_2\mathbin{\star}\Z_2$ is
freely generated by
\[
  (g_2,g_1),\;(g_3,g_1),\;(g_3,g_2),\;(g_1,(g_3,g_2)),\;(g_2,(g_3,g_1)).
\]
This is easy to see geometrically, by identifying $\racg'_\sK$
with $\pi_1(\rk)$. In our case $\rk$ is the 1-skeleton of the
3-cube (see Example~\ref{ppexa}.4). We have
$(g_1,(g_3,g_2))=g_1(g_2,g_3)g_1(g_3,g_2)$,
$(g_2,(g_3,g_1))=g_2(g_1,g_3)g_2(g_3,g_1)$, and the elements
$(g_2,g_1)$, $(g_3,g_1)$, $(g_3,g_2)$, $g_1(g_2,g_3)g_1$,
$g_2(g_1,g_3)g_2$ correspond to the loops around five different
faces of $\rk$, which freely generate its fundamental group.

The general statement for arbitrary $m$ can be proved by a similar
topological argument, by identifying $G'$ with the fundamental
group of the 1-skeleton of the $m$-dimensional cube
(see~\cite[Proposition~3.6]{staf15}). However, we record an
algebraic argument for subsequent use. We have the commutator
identity
\begin{equation}\label{swap}
  \!(g_q,(g_p,x))\!=\!(g_q,x)(x,(g_p,g_q))(g_q,g_p)(x,g_p)(g_p,(g_q,x))
  (x,g_q)(g_p,g_q)(g_p,x),
\end{equation}
which can be deduced from the Hall--Witt identity, or checked
directly. Note that if $x$ is a canonical nested commutator, then
the factor $(x,(g_p,g_q))$ can be expressed via nested commutators
as in the beginning of the proof of Theorem~\ref{gscox}. In this
case we can use~\eqref{swap} to swap $g_p$ and $g_q$ in the
commutator $(g_q,(g_p,x))$, by expressing it through
$(g_p,(g_q,x))$ and canonical nested commutators of lesser length.

In the subsequent arguments, we shall swap elements in an iterated
commutator. Such a swap will change the element of the group
represented by the commutator, but the two elements will always
differ by a product of commutators of lesser length, as in the
case of $(g_p,(g_q,x))$ and $(g_q,(g_p,x))$ in the argument above.
If we want to swap elements $g_p$ and $g_q$ inside a canonical nested commutator of 
the form $(\cdots,(g_p,(g_q,x))\cdots)$, where $x$ is a smaller canonical nested commutator,
then we need to use the first identity of~\eqref{hall} alongside with~\eqref{swap}. Note that
if both $b$ and $c$ in the identity $(a,bc)=(a,c)(a,b)((a,b),c)$ are canonical nested commutators,
then $(a,c)$ and $(a,b)$ are also canonical nested commutators, while $((a,b),c)=(b,a)c^{-1}(a,b)c$
is a product of nested commutators of lesser length.

Therefore, using~\eqref{swap} together with the
identities~\eqref{hall} we can change the order of any two
generators in a commutator
$(g_{i_1},\cdots(g_{i_{\ell-2}},(g_{i_{\ell-1}},g_{i_\ell}))\cdots)$
within the positions $i_1$ to~$i_{\ell-2}$. We first use this
observation to eliminate commutators
$(g_{i_1},\cdots(g_{i_{\ell-2}},(g_{i_{\ell-1}},g_{i_\ell}))\cdots)$
which contain a pair of repeating generators~$g_i$ (i.\,e. have
$i_p=i_q$ for some $p\ne q$). Namely, if the repeating pair occurs
within the positions $i_1$ to~$i_{\ell-2}$, then we
use~\eqref{swap} to reduce the commutator to the form
$(\cdots(g_i,(g_i,x))\cdots)$, where
$x=(g_{i_{\ell-1}},g_{i_\ell})$, and use the relation
$(g_i,(g_i,x))=(g_i,x)(g_i,x)$ to reduce the commutator to a
product of commutators of lesser length. (Note that here we use
the relation $g_i^2=1$ in~$G$.) If one of the repeating generators
is on the position $i_{\ell-1}$ or $i_\ell$, then we
use~\eqref{swap} to reduce the commutator to the form
$(\cdots(g_i,(g_i,g_j))\cdots)$ and use the relation
$(g_i,(g_i,g_j))=(g_i,g_j)(g_i,g_j)$. As a result, we obtain a
generator set for $G'$ consisting of commutators
$(g_{i_1},\cdots(g_{i_{\ell-2}},(g_{i_{\ell-1}},g_{i_\ell}))\cdots)$
with all different~$g_i$. This already shows that $G'$ is finitely
generated.

Now we use \eqref{swap} to put the generators $g_i$ in
$(g_{i_1},\cdots(g_{i_{\ell-2}},(g_{i_{\ell-1}},g_{i_\ell}))\cdots)$
in an order. Choose the generator $g_{i_k}$ with the largest index
$i_k$. If it is not within the last three positions, then we use
\eqref{swap} to move it to the third-to-last position. The case
$m=3$ considered above shows that the commutator $(g_j,(g_i,g_k))$
can be expressed through $(g_i,(g_j,g_k)$, $(g_k,(g_j,g_i))$ and
commutators of lesser length. This allows us to move the generator
$g_{i_k}$ with the largest index $i_k$ in
$(g_{i_1},\cdots(g_{i_{\ell-2}},(g_{i_{\ell-1}},g_{i_\ell}))\cdots)$
to the second-to-last position, and set $j=i_k$. Then we
use~\eqref{swap} and~\eqref{hall} to put the first $\ell-2$
generators in the commutator in an increasing order, and redefine
their indices as $k_1<\cdots<k_{\ell-2}$. As a result, we obtain a
generator set for $G'$ consisting of commutators of the required
form $(g_{k_1},(g_{k_2},\cdots(g_{k_{\ell-2}},(g_j,g_i))\cdots))$
where $k_1<k_2<\cdots<k_{\ell-2}<j>i$ and $k_s\ne i$ for any~$s$.

It remains to show that the constructed generating set of $G'$ is
free. This generating set consists of
$N=\sum_{\ell=2}^m(\ell-1)\binom m\ell=(m-2)2^{m-1}+1$
commutators. On the other hand, $G'\cong\pi_1(\rk)$, where $\rk$
the 1-skeleton of the $m$-cube. Then $\rk$ is homotopy equivalent
to a wedge of $N$ circles (as easy to see inductively or by
computing the Euler characteristic), and $\pi_1(\rk)$ is a free
group of rank $N$. We therefore have a system of $N$ generators in
a free group of rank~$N$. This system must be free, by the
classical result that a free group of finite rank cannot be
isomorphic to its proper quotient group,
see~\cite[Theorem~2.13]{m-k-s76}.
\end{proof}

Now we resume the proof of Theorem~\ref{gscox}. We shall exclude
some commutators
$(g_{k_1},(g_{k_2},\cdots(g_{k_{\ell-2}},(g_j,g_i))\cdots))$ from
the generating set using the new commutativity relations.

First assume that $j$ and $i$ are in the same connected component
of the complex $\sK_{\{k_1,\ldots,k_{\ell-2},j,i\}}$. We shall
show that the corresponding commutator
$(g_{k_1},(g_{k_2},\cdots(g_{k_{\ell-2}},(g_j,g_i))\cdots))$ can
be excluded from the generating set. We choose a path from $i$ to
$j$, i.\,e. choose vertices $i_1,\ldots,i_n$ from
$k_1,\ldots,k_{\ell-2}$ with the property that the edges
$\{i,i_1\},\{i_{1},i_{2}\},\ldots,\{i_{n-1},i_{n}\},\{i_n,j\}$ are
all in~$\sK$. We proceed by induction on the length of the path.
The induction starts from the commutator $(g_j,g_i)=1$
corresponding to a one-edge path $\{i,j\}\in\sK$. Now assume that
the path consists of $n+1$ edges. Using the relation~\eqref{swap}
we can move the elements $g_{i_1},g_{i_2},\ldots,g_{i_n}$ in
$(g_{k_1},(g_{k_2},\cdots(g_{k_{\ell-2}},(g_j,g_i))\cdots))$ to
the right and restrict ourselves to the commutator
$(g_{i_1},(g_{i_2},\cdots(g_{i_n},(g_j,g_i))\cdots))$. Observe
that in the presence of the commutation relation $(g_p,g_q)=1$ the
identity~\eqref{swap} does not contain the factor $(x,(g_p,g_q))$
and therefore it allows us to change the order of $g_p$ and $g_q$
without assuming $x$ to be a commutator. We therefore can convert
the commutator
$(g_{i_1},(g_{i_2},\cdots(g_{i_n},(g_j,g_i))\cdots))$ (with
$\{i_n,j\}\in\sK$) to the commutator
$(g_j,(g_{i_1},\cdots(g_{i_{n-1}},(g_{i_n},g_i))\cdots))$. The
latter contains a commutator
$(g_{i_1},\cdots(g_{i_{n-1}},(g_{i_n},g_i))\cdots)$ corresponding
to a shorter path $\{i,i_1,\ldots,i_n\}$. By inductive hypothesis,
it can be expressed through commutators of lesser length, and
therefore excluded from the set of generators.

We therefore obtain a generator set for $\racg_\sK'$ consisting of
nested commutators
$(g_{k_1},\cdots(g_{k_{\ell-2}},(g_j,g_i))\cdots)$ with $j$ and
$i$ in different connected components of the complex
$\sK_{\{k_1,\ldots,k_{\ell-2},j,i\}}$. Consider commutators
$(g_{k_1},\cdots(g_{k_{\ell-2}},(g_j,g_{i_1}))\cdots)$ and
$(g_{k'_1},\cdots(g_{k'_{\ell-2}},(g_j,g_{i_2}))\cdots)$ such that
$\{k_1,\ldots,k_{\ell-2},j,i_1\}=\{k'_1,\ldots,k'_{\ell-2},j,i_2\}$
and $i_1,i_2$ lie in the same connected component of
$\sK_{\{k_1,\ldots,k_{\ell-2},j,i_1\}}$ which is different from
the connected component containing~$j$. We claim that one of these
commutators can be expressed through the other and commutators of
lesser length. To see this, we argue as in the previous paragraph,
by considering a path in $\sK_{\{k_1,\ldots,k_{\ell-2},j,i_1\}}$
between $i_1$ and~$i_2$, and then reducing it inductively to a
one-edge path. This leaves us with the pair of commutators
$(g_{i_2},(g_j,g_{i_1}))$ and $(g_{i_1},(g_j,g_{i_2}))$ where
$\{i_1,i_2\}\in\sK$, $\{i_1,j\}\notin\sK$, $\{i_2,j\}\notin\sK$.
The claim then follows easily from the relation
$(g_{i_1},g_{i_2})=1$ (compare the case $m=3$ of
Lemma~\ref{fpgen}).

Thus, to enumerate independent commutators, we use the convention
of writing $(g_{k_1},\cdots(g_{k_{\ell-2}},(g_j,g_i))\cdots)$
where $i$ is the smallest vertex in its connected component within
$\sK_{\{k_1,\ldots,k_{\ell-2},j,i\}}$. This leaves us with
precisely the set of commutators~\eqref{commuset}. It remains to
show that this generating set is minimal. For this we once again
recall that $\racg_\sK'=\pi_1(\rk)$. The first homology group
$H_1(\rk)$ is isomorphic to $\racg_\sK'/\racg_\sK''$, where
$\racg_\sK''$ is the commutator subgroup of~$\racg_\sK'$. On the
other hand, we have
\[
  H_1(\rk)\cong\bigoplus_{J\subset[m]}\widetilde H_0(\sK_J)
\]
by Theorem~\ref{homrk}. Hence, the number of generators in the
abelian group $H_1(\rk)\cong \racg_\sK'/\racg_\sK''$ is
$\sum_{J\subset[m]}\rank\widetilde H_0(\sK_J)$. The latter number
agrees with the number of iterated commutators in the
set~\eqref{commuset}, as $\rank\widetilde H_0(\sK)$ is one less
the number of connected components of~$\sK$.
\end{proof}

\begin{example}\

\textbf{1.} Let $\mathcal K$ be the simplicial complex
$\quad\begin{picture}(10,0) \put(0,0){\circle*{1}}
\put(5,0){\circle*{1}} \put(5,5){\circle*{1}}
\put(0,0){\line(1,0){5}} \put(5,0){\line(0,1){5}}
\put(-2.5,-1){\scriptsize 1} \put(6.3,-1){\scriptsize 2}
\put(6.3,4){\scriptsize 3} \put(10,0){\circle*{1}}
\put(11.3,-1){\scriptsize 4}
\end{picture}\quad$ on four vertices. Then the commutator subgroup $\racg'_\sK$ is free, and
Theorem~\ref{gscox} gives the following free generators:
\begin{gather*}
  (g_3,g_1),\;(g_4,g_1),\;(g_4,g_2),\;(g_4,g_3),\\
  (g_2,(g_4,g_1)),\;(g_3,(g_4,g_1)),\;(g_1,(g_4,g_3)),\;(g_3,(g_4,g_2)),\\
  (g_2,(g_3,(g_4,g_1))).
\end{gather*}

\textbf{2.} Let $\sK$ be an $m$-cycle with $m\ge4$ vertices. Then
$\sK^1$ is not a chordal graph, so the group $\racg'_\sK$ is not
free. One can see that $\rk$ is an orientable surface of genus
$(m-4)2^{m-3}+1$ (see Example~\ref{mcyc1}), so
$\racg'_\sK\cong\pi_1(\rk)$ is a one-relator group.

When $m=4$, we get a 2-torus as $\rk$, and Theorem~\ref{gscox}
gives the generators $a_1=(g_3,g_1)$ and $b_1=(g_4,g_2)$. The
single relation is obviously $(a_1,b_1)=1$. For $m\ge5$ we do not
know the explicit form of the single relation in the surface group
$\racg'_\sK\cong\pi_1(\rk)$ in terms of the generators provided by
Theorem~\ref{gscox}. Compare~\cite{very}, where the corresponding
problem is studied for the commutator subalgebra of the graded Lie
algebra from Theorem~\ref{multgen}.
\end{example}

\end{document}